\documentclass[a4paper,12pt,twoside]{amsart}
\usepackage[utf8]{inputenc}
\usepackage{latexsym}
\usepackage[pdftex]{graphicx}
\usepackage{hyperref}
\usepackage{amsmath,amssymb, amsthm, amscd}
\usepackage{dsfont}
\usepackage{multicol}
\usepackage[marginratio=1:1,height=597pt,width=460pt,tmargin=117pt]{geometry}
\usepackage[all]{xy}
\usepackage{etoolbox}
\usepackage{enumitem}
\usepackage{tikz-cd}
\usepackage{mathtools}
\usepackage{float}
\usepackage{comment}
\usepackage{array}
\usepackage{bm}
\usepackage[pagewise]{lineno}
\expandafter\patchcmd\csname\string\proof\endcsname
  {\normalparindent}{0pt }{}{}
\makeatletter
\patchcmd{\@thm}{\thm@headfont{\scshape}}{\thm@headfont{\scshape\bfseries}}{}{}
\patchcmd{\@thm}{\thm@notefont{\fontseries\mddefault\upshape}}{}{}{}
\makeatother
\makeatletter
\patchcmd\@thm
  {\let\thm@indent\indent}{\let\thm@indent\indent}%
  {}{}
\makeatother
\newtheorem*{theorem*}{Theorem}
\newtheorem{theorem}[equation]{Theorem}
\newtheorem{lemma}[equation]{Lemma}
\newtheorem{proposition}[equation]{Proposition}

\theoremstyle{definition}

\theoremstyle{definition}
\newtheorem{remark}[equation]{Remark}
\theoremstyle{remark}

\numberwithin{subsection}{section}
\numberwithin{equation}{section}
\newcommand{\iso}{\xrightarrow{
   \,\smash{\raisebox{-0.50ex}{\ensuremath{\scriptstyle\sim}}}\,}}
\title[Non-admissible irreducible representations of $\mathrm{GL}_{n}$]{Non-admissible irreducible representations of $p$-adic $\mathrm{GL}_{n}$ in characteristic $p$}
\author{Eknath Ghate, Daniel Le, Mihir Sheth}
\address{School of Mathematics, Tata Institute of Fundamental Research \\ Homi Bhabha Road, Mumbai - 400005, India.}
\email{eghate@math.tifr.res.in}
\address{Department of Mathematics, Purdue University, 150 N. University Street, West Lafayette, IN 47907, USA.}
\email{ledt@purdue.edu}
\address{Department of Mathematics, Indian Institute of Science \\ Bangalore - 560012, India.}
\email{mihirsheth@iisc.ac.in}
\begin{document}

\begin{abstract}
	Let $p>3$ and $F$ be a non-archimedean local field with residue field a proper finite extension of $\mathbb{F}_p$. We construct smooth absolutely irreducible non-admissible representations of $\mathrm{GL}_2(F)$ defined over the residue field of $F$ extending the earlier results of the authors for $F$ unramified over $\mathbb{Q}_{p}$. This construction uses the theory of diagrams of Breuil and Pa\v sk$\bar{\mathrm{u}}$nas. By parabolic induction, we obtain smooth absolutely irreducible non-admissible representations of $\mathrm{GL}_n(F)$ for $n>2$. 
\end{abstract}

\maketitle
\section{Introduction}
Let $p$ be a prime number. 
This note concerns the smooth representation theory of (connected) $p$-adic reductive groups over coefficient fields of characteristic $p$ initiated in \cite{BL}. 
This theory has its origins in the study of congruences between automorphic forms and plays an important role in the mod $p$ Langlands program proposed by Breuil \cite{Breuil}. 
In our context, smooth means that the stabilizers of vectors are open subgroups. 
Spaces of automorphic forms provide natural sources of smooth representations which are also \emph{admissible}, i.e.,~the space of vectors invariant under any compact open subgroup is finite-dimensional. 
Over characteristic $0$ fields, building upon Harish-Chandra's work \cite{HC}, Jacquet \cite{Jacquet} and Bernstein \cite{Bernstein} showed that any irreducible (or finite length) smooth representation of a $p$-adic reductive group is automatically admissible by reducing to the supercuspidal case. 
Vign\'eras extended this result to base fields of positive characteristic different from $p$ \cite{Vigneras}. 
The proofs use Haar measures which do not exist in characteristic $p$. 
Nevertheless, \cite[Question 1]{ahhv} asked whether a similar result holds in characteristic $p$. 
It is not hard to see that smooth irreducible representations of $p$-adic reductive groups which are anisotropic modulo center are finite-dimensional. 
Berger showed that any irreducible representation of $\mathrm{GL}_2(\mathbb{Q}_p)$ over an algebraically closed field of characteristic $p$ admits a central character \cite{Berger}. Barthel-Livn\'{e} and Breuil classified the irreducible representations of $\mathrm{GL}_2(\mathbb{Q}_p)$ over an algebraically closed field of characteristic $p$ with central character \cite{BL, Breuil} and a direct computation shows that each such representation is admissible. Together these results imply that any absolutely irreducible representation of $\mathrm{GL}_2(\mathbb{Q}_p)$ over a field of characteristic $p$ is admissible.
Recently, the authors \cite{Le,GS} used the theory of diagrams developed by Breuil and Pa\v sk$\bar{\mathrm{u}}$nas \cite{Paskunas,BP} to construct absolutely irreducible smooth representations of $\mathrm{GL}_2(F)$ in characteristic $p$ which are \emph{not} admissible when $F$ is a proper finite unramified extension of $\mathbb{Q}_p$ and $p>2$ (see also \cite{Ghate2022}). 
This naturally leads one to ask which $p$-adic reductive groups admit irreducible non-admissible representations. 
Here, we focus on the case of $\mathrm{GL}_n(F)$. 

\begin{theorem} \label{thm:intro}
Let $p>3$ and $n\geq 2$. Let $F$ be a non-archimedean local field with residue field a proper finite extension of $\mathbb{F}_p$. 
Then there is an absolutely irreducible non-admissible smooth representation of $\mathrm{GL}_n(F)$ defined over the residue field of $F$. 
\end{theorem}

\noindent The hypothesis in Theorem \ref{thm:intro} that the residue field of $F$ is not $\mathbb{F}_p$ cannot be entirely removed given the results of Berger, Barthel-Livn\'{e}, and Breuil above (see also Remark \ref{main-idea-rmk}). Following the methods of \cite{Le}, we also have a counterexample to a Schur-type lemma for irreducible representations of $\mathrm{GL}_2(F)$. 

\begin{theorem}\label{thm:schur}
Let $p>3$ and $F$ be a non-archimedean local field with residue field a proper finite extension of $\mathbb{F}_p$. 
Then there is an irreducible smooth representation of $\mathrm{GL}_2(F)$ over the residue field of $F$ whose endomorphism algebra contains an algebraically closed field. 
\end{theorem}

We prove Theorem \ref{thm:intro} by first constructing smooth absolutely irreducible non-admissible representations for $\mathrm{GL}_2(F)$. The construction is uniform and provides a new construction in the cases when $F$ is an unramified extension of $\mathbb{Q}_p$. By parabolically inducing non-admissible irreducible representations of $\mathrm{GL}_{2}(F)$, we obtain such representations of $\mathrm{GL}_n(F)$ for $n>2$. The proof of the irreducibility of induced representations uses Herzig's comparison isomorphism between compact and parabolic inductions. We remark that the non-admissible irreducible representations constructed here have a central character. The ones for $\mathrm{GL}_2(F)$ are necessarily supersingular by the classification of Barthel-Livn\'{e}. The ones for $\mathrm{GL}_n(F)$ with $n>2$ are, by contrast, not supersingular.

The reason for restricting to unramified extensions of $\mathbb{Q}_p$ in our earlier works is that we used some of the results of \cite{BP} relying on delicate Witt vector computations to prove the irreducibility. 
Recently, one of us \cite{she22} introduced cyclic modules to circumvent the irreducibility arguments of \cite{BP} and construct infinitely many supercuspidal representations of $\mathrm{GL}_2(F)$ with fixed central character under the assumptions in Theorem \ref{thm:intro}. 
Our construction of an irreducible non-admissible representation of $\mathrm{GL}_2(F)$ involves splicing two cyclic modules together. 
The resulting diagram is quite different from the diagrams appearing in \cite{BP,Le,GS}, namely the $\mathrm{GL}_2(\mathcal{O}_F)$-subrepresentation generated by a pro-$p$ Iwahori fixed eigenvector can have reducible socle. 
This construction was inspired by similar features of the mod $p$ cohomology of $U(3)$ arithmetic manifolds (see \cite{LLLM}). 
Finally, one of the motivations for our construction is a recent conjecture of Emerton, Gee, Hellmann, and Zhu \cite[Conjecture 2.4.3]{EGH} stating that there should exist a fully faithful functor from the category of smooth representations of $\mathrm{GL}_n(F)$ to the category of quasicoherent sheaves on an appropriate moduli stack of Langlands parameters. 
The existence of irreducible non-admissible smooth $\mathrm{GL}_n(F)$-representations should have an interpretation in terms of the geometry of this moduli stack. We hope to return to this in future work. 

\subsection*{Acknowledgements} We thank the anonymous referee for many helpful comments and corrections on an earlier version of this paper. During this work, the second-named author was supported by a start-up grant from Purdue University, and the third-named author was supported by the Raman Postdoctoral Fellowship from Indian Institute of Science, Bangalore. We also like to thank the organizers of the International Centre for Theoretical Sciences (ICTS) program ``Elliptic curves and the special values of $L$-functions" (code:ICTS/ecl2022/8) for their invitation and the hospitality during which the collaboration on this project began.

\subsection*{Notation and convention} Let $p>3$ be a prime number. Let $\overline{\mathbb{F}}_{p}$ be the algebraic closure of the finite field $\mathbb{F}_{p^{f}}$ of size $p^{f}$. Fix an embedding $\mathbb{F}_{p^{f}}\hookrightarrow\overline{\mathbb{F}}_{p}$. Let $F$ be a non-archimedean local field of residual characteristic $p$ and residue degree $f>1$. Let $\mathcal{O}_{F}\subseteq F$ be the valuation ring with a uniformizer $\varpi$. Throughout the note, except for the last part, we work with the group $\mathrm{GL}_{2}(F)$. Let $G=\mathrm{GL}_{2}(F)$, $K=\mathrm{GL}_{2}(\mathcal{O}_{F})$, $\Gamma=\mathrm{GL}_{2}(\mathbb{F}_{p^{f}})$, and $Z$ be the center of $G$. Let $B$ and $U$ be the subgroups of $\Gamma$ consisting of the upper triangular matrices and the upper triangular unipotent matrices respectively. Let $I$ and $I(1)$ be the preimages of $B$ and $U$ respectively under the reduction modulo $\varpi$ map $K\twoheadrightarrow\Gamma$. The subgroups $I$ and $I(1)$ of $K$ are the Iwahori and the pro-$p$ Iwahori subgroup of $K$ respectively. The normalizer $N$ of $I$ in $G$ is a subgroup generated by $I$ and $\Pi=\left(\begin{smallmatrix}
0 & 1 \\ \varpi & 0 \end{smallmatrix} \right)$. Note that $N$ is also the normalizer of $I(1)$ in $G$. Let $K(1)$ denote the kernel of the map $K\twoheadrightarrow\Gamma$, i.e., the first principal congruence subgroup of $K$. Unless stated otherwise, all representations considered in this note are on $\overline{\mathbb{F}}_{p}$-vector spaces. 

A \emph{weight} is an irreducible representation of $\Gamma$. Any weight is of the form of \[\left(\bigotimes\limits_{j=0}^{f-1}\mathrm{Sym}^{r_{j}}\overline{\mathbb{F}}_{p}^{2}\circ\Phi^{j}\right)\otimes\mathrm{det}^{m}\] for some integers $0\leq r_{0},\ldots,r_{f-1}\leq p-1$ and $0\leq m\leq p^{f}-2$, where $\Phi:\Gamma\rightarrow\Gamma$ is the automorphism induced by the Frobenius map $\alpha\mapsto \alpha^{p}$ on $\mathbb{F}_{p^{f}}$ and $\mathrm{det}:\Gamma\rightarrow\mathbb{F}_{p^{f}}^{\times}$ is the determinant character. We denote such a weight by $\bm{r}\otimes\mathrm{det}^{m}$ where $\bm{r}$ is the $f$-tuple $(r_{0},\ldots,r_{f-1})$ of integers. Let $\sigma=\bm{r}\otimes\mathrm{det}^{m}$ be a weight; its subspace $\sigma^{U}$ of $U$-fixed vectors is $1$-dimensional and stable under the action of $B$ because $B$ normalizes $U$. The resulting $B$-character, denoted by $\chi(\sigma)$, sends $\left(\begin{smallmatrix}a & b \\ 0 & d \end{smallmatrix} \right)\in B$ to $a^{r}(ad)^{m}$ where $r=\sum_{j=0}^{f-1}r_{j}p^{j}$. Any $B$-character valued in $\overline{\mathbb{F}}_{p}^{\times}$ factors through the quotient $B/U$ which is identified with the subgroup of diagonal matrices in $B$ by the section $B/U\rightarrow B$, $\left(\begin{smallmatrix}a & 0 \\ 0 & d \end{smallmatrix} \right)U\mapsto\left(\begin{smallmatrix}a & 0 \\ 0 & d \end{smallmatrix} \right)$. For a $B$-character $\chi$, let $\chi^{s}$ be the inflation to $B$ of the conjugation-by-$s$ character $t\mapsto\chi(sts^{-1})$ on $B/U$ where $s=\left(\begin{smallmatrix}0 & 1 \\ 1 & 0 \end{smallmatrix} \right)$. We say that a weight is \emph{generic} if it is not equal to $(0,0,\ldots,0)\otimes\mathrm{det}^{m}$ or $(p-1,p-1,\ldots,p-1)\otimes\mathrm{det}^{m}$ for any $m$. The map $\sigma\mapsto\chi(\sigma)$ gives a bijection from the set of generic weights to the set of $B$-characters $\chi$ such that $\chi\neq\chi^{s}$. If $\sigma$ is a generic weight, let us denote by $\sigma^{[s]}$ the generic weight corresponding to the character $\chi(\sigma)^{s}$. For $\sigma=\bm{r}\otimes\mathrm{det}^{m}$, we have $\sigma^{[s]}=(p-1-r_{0},\ldots,p-1-r_{f-1})\otimes\mathrm{det}^{m+r}$. For a $B$-representation $V$ and a character $\chi$, we denote by $V^{\chi}$ the $\chi$-isotypic component of $V$. We refer the reader to \cite[\S 1]{BL} for all non-trivial assertions in this paragraph.

Given two weights $\sigma$ and $\tau$, let $E(\sigma, \tau)$ be the unique non-split $\Gamma$-extension \[0\longrightarrow\sigma\longrightarrow E(\sigma, \tau)\longrightarrow\tau\longrightarrow 0\] of $\tau$ by $\sigma$ whenever it exists \cite[Corollary 5.6]{BP}. A finite-dimensional representation of $\Gamma$ is said to be \emph{multiplicity-free} if the multiset of its Jordan-H\"{o}lder factors is multiplicity-free. For any group $H$, the socle and the cosocle of an $H$-representation $\pi$ are denoted by $\mathrm{soc}_{H}\pi$ and $\mathrm{cosoc}_{H}\pi$ respectively.   

Note that a weight is a smooth irreducible representation of $K$ (resp. of $KZ$) and a $B$-character is a smooth $I$-character (resp. $IZ$-character) via the map $K\twoheadrightarrow\Gamma$ (resp. $KZ\twoheadrightarrow\Gamma$). In fact, the weights exhaust all smooth irreducible representations of $K$ (resp. of $KZ$ such that $\varpi$ acts trivially). In the last section, we also talk of $\mathrm{M}(\mathcal{O}_{F})$-weights for a Levi subgroup $\mathrm{M}\subseteq\mathrm{GL}_{n}$ which mean smooth irreducible representations of $\mathrm{M}(\mathcal{O}_{F})$. 

\section{The spliced module}

We recall some notation from \cite[\S 1]{she22} that is used in this section. Let $\left(\mathbb{Z}\pm x\right)^{f}$ be the set of $f$-tuples of linear polynomials in $x$ having integral coefficients with leading coefficient $\pm 1$. For $\bm{\lambda}=(\lambda_{0}(x),\ldots,\lambda_{f-1}(x))$ and $\bm{\lambda}'=(\lambda_{0}'(x),\ldots,\lambda_{f-1}'(x))\in(\mathbb{Z}\pm x)^{f}$, let \begin{linenomath*}
	\[\bm{\lambda}\circ\bm{\lambda}':=(\lambda_{0}(\lambda_{0}'(x)),\ldots,\lambda_{f-1}(\lambda_{f-1}'(x)))\in(\mathbb{Z}\pm x)^{f}.\]
\end{linenomath*}
Let $\bm{\mu}\in(\mathbb{Z}\pm x)^{f}$ be the $f$-tuple of polynomials defined by 
\begin{linenomath*}
\begin{align*}
&\mu_{0}(x):=x-1,\nonumber\\&\mu_{1}(x):=p-2-x,\\&\mu_{j}(x):=p-1-x 
\hspace{2mm}\text{for $2\leq j\leq f-1$}.\nonumber
\end{align*}
\end{linenomath*}
When $f=2$, the condition $2\leq j\leq f-1$ is empty and $\bm{\mu}=(\mu_{0}(x),\mu_{1}(x))=(x-1,p-2-x)$. Let $g\in S_{f}$ denote the cyclic permutation $(f(f-1)\ldots 21)$, and let 
\begin{linenomath*}
\[\bm\mu^{(k)}:=g^{k-1}\bm{\mu}\circ g^{k-2}\bm{\mu}\circ\ldots\circ g\bm{\mu}\circ\bm\mu \hspace{2mm} \text{for all $1\leq k\leq l$},\]
\end{linenomath*}
where $l$ is equal to $f$ (resp. $2f$) if $f$ is odd (resp. even). We let $\bm\mu^{(0)}=(x,x,\ldots,x)$. It follows from the definition of $\bm\mu^{(k)}$ that, for $1\leq k\leq l$, 
\begin{equation}\label{inductiveformulae}
\mu_{j}^{(k)}(x)=
\begin{cases}
\mu_{j}^{(k-1)}(x)-1& \text{if $j\equiv 1-k\mod{f}$,}\\
p-2-\mu_{j}^{(k-1)}(x)& \text{if $j\equiv 2-k\mod{f}$,}\\
p-1-\mu_{j}^{(k-1)}(x) & \text{otherwise}.
\end{cases}
\end{equation} 
Recall from \cite[Lemma 1.4 (1)]{she22} that $\bm\mu^{(l)}=\bm\mu^{(0)}=(x,x,\ldots,x)$.
We assign to $\bm\mu^{(k)}$ an element $\bm{m}^{(k)}\in(\mathbb{Z}/2\mathbb{Z})^{f}$ according to the rule that its $j$-th entry $m^{(k)}_{j}$ is $0$ if and only if the sign of $x$ in $\mu_{j}^{(k)}(x)$ is $+$. 
\begin{lemma}\label{relation_betn_mk_and_ml-k}
	\begin{enumerate}
		\item For all $1\leq k \leq l$, $\bm{m}^{(k)}=g^{k}\bm{m}^{(l-k)}$.
		\item For $1\leq k_{1},k_{2}\leq l-1$ and $k_{1}\neq k_{2}$, $\bm{m}^{(k_{1})}$ and $\bm{m}^{(k_{2})}$ are (cyclic) permutations of each other if and only if $k_{2}=l-k_{1}$.
		\item For $1\leq k\leq l-1$, $k\neq \frac{l}{2}$ if $f$ is even, $\bm{m}^{(k)}$ is not equal to any of its non-trivial cyclic permutations.
	\end{enumerate} 
\end{lemma}
\begin{proof}
(1) By definition, $\bm{m}^{(k)}=\sum_{i=0}^{k-1}g^{i}\bm{m}^{(1)}$. Since $\bm{m}^{(l)}=(0,0,\ldots,0)$, we have \[\sum_{i=0}^{l-1}g^{i}\bm{m}^{(1)}=(0,0,\ldots,0).\] Thus, \[\sum_{i=0}^{k-1}g^{i}\bm{m}^{(1)}+g^{k}\sum_{i=0}^{l-k-1}g^{i}\bm{m}^{(1)}=(0,0,\ldots,0),\hspace{2mm}\text{i.e.,}\hspace{2mm}\bm{m}^{(k)}+g^{k}\bm{m}^{(l-k)}=(0,0,\ldots,0).\] Since an element of $(\mathbb{Z}/2\mathbb{Z})^{f}$ is equal to its additive inverse, (1) follows. 

(2) If $\bm{m}^{(k_{1})}$ and $\bm{m}^{(k_{2})}$ are (cyclic) permutations of each other for $1\leq k_{1},k_{2}\leq l-1$, then the tuples $\bm{m}^{(k_{1})}$ and $\bm{m}^{(k_{2})}$ have the same number of $0$'s. When $f$ is odd (resp. even), the number of $0$'s in $\bm{m}^{(k)}$ for odd $k$ equals $k$ (resp. $k$ if $k\leq \frac{l}{2}$ and $l-k$ if $k>\frac{l}{2}$), and the number of $0$'s in $\bm{m}^{(k)}$ for even $k$ equals $l-k$ (resp. $\frac{l}{2}-k$ if $k\leq \frac{l}{2}$ and $k-\frac{l}{2}$ if $k>\frac{l}{2}$). Hence, it follows that if $\bm{m}^{(k_{1})}$ and $\bm{m}^{(k_{2})}$ are (cyclic) permutations of each other, then either $k_{1}=k_{2}$ or $k_{2}=l-k_{1}$. This proves the forward implication. The converse statement follows from (1). 

(3) By (1), it is enough to show (3) for $1\leq k\leq f-1$. Now, (3) follows from the observation that for $1\leq k\leq f-1$, the tuple $\bm{m}^{(k)}$ is a cyclic permutation of a tuple of the form $k$ $0$'s followed by $(f-k)$ $1$'s (resp. $(f-k)$ $0$'s followed by $k$ $1$'s) for odd (resp. even) $k$. 
\end{proof}

\begin{lemma}\label{multiplicity_free_lemma} $\lbrace\bm{\mu}^{(1)},\bm{\mu}^{(2)},\ldots,\bm{\mu}^{(l-1)},\bm{\mu}^{(l)},g\bm{\mu}^{(1)},g\bm{\mu}^{(2)},\ldots,g\bm{\mu}^{(l-1)}\rbrace$ is a set of distinct $f$-tuples in $(\mathbb{Z}\pm x)^{f}$.
\end{lemma}
\begin{proof}
By \cite[Lemma 1.4 (2)]{she22}, it is enough to prove that $\bm\mu^{(k_{1})}\neq g\bm\mu^{(k_{2})}$ for $1\leq k_{1},k_{2}\leq l-1$. If $\bm\mu^{(k_{1})}=g\bm\mu^{(k_{2})}$ for some $1\leq k_{1},k_{2}\leq l-1$, then we have $\bm{m}^{(k_{1})}=g\bm{m}^{(k_{2})}$ for the corresponding elements in $(\mathbb{Z}/2\mathbb{Z})^{f}$. We now find all the pairs $(k_{1},k_{2})$ satisfying $\bm{m}^{(k_{1})}=g\bm{m}^{(k_{2})}$. If $k_{1}=k_{2}=k$, then $\bm{m}^{(k)}=g\bm{m}^{(k)}$. By Lemma \ref{relation_betn_mk_and_ml-k} (3), it follows that  $f$ is even and $k=f=\frac{l}{2}$. If $k_{1}\neq k_{2}$, we use Lemma \ref{relation_betn_mk_and_ml-k} (1) and (2) to find that $\bm{m}^{(l-k_{1})}=g^{k_{1}-1}\bm{m}^{(l-k_{1})}$. By Lemma \ref{relation_betn_mk_and_ml-k} (3), $g^{k_{1}-1}$ must be the identity permutation. This gives $k_{1}=1$ (resp. $k_{1}=1$ or $\frac{l}{2}+1$) for odd (resp. even) $f$.
Therefore, the pairs $(k_{1},k_{2})$ satisfying $\bm{m}^{(k_{1})}=g\bm{m}^{(k_{2})}$ are
\begin{enumerate}
\item $(1,l-1)$ if $f$ is odd,
\item $(1,l-1)$, $(\frac{l}{2}+1,\frac{l}{2}-1)$, $(\frac{l}{2},\frac{l}{2})$ if $f$ is even.
\end{enumerate}
In Case (1), one checks using (\ref{inductiveformulae}) that $\mu^{(1)}_{0}(x)=x-1\neq x+1=\mu^{(l-1)}_{1}(x)$. Thus $\bm\mu^{(1)}\neq g\bm\mu^{(l-1)}$.
In Case (2), one checks using (\ref{inductiveformulae}) again that $\mu^{(1)}_{0}(x)=x-1\neq x+1=\mu^{(l-1)}_{1}(x)$ in the first subcase, $\mu^{(\frac{l}{2}+1)}_{1}(x)=x+1\neq x-1=\mu^{(\frac{l}{2}-1)}_{2}(x)$ in the second subcase, and $\mu^{(\frac{l}{2})}_{0}(x)=p-1-x\neq p-3-x=\mu^{(\frac{l}{2})}_{1}(x)$ in the third subcase.
\end{proof}
For $\bm\lambda=(\lambda_{0}(x),\ldots,\lambda_{f-1}(x))\in\left(\mathbb{Z}\pm x\right)^{f}$ and $\bm{r}\in\mathbb{Z}^{f}$, \[\bm\lambda(\bm{r}):=\left(\lambda_{0}(r_{0}),\lambda_{1}(r_{1}),\ldots,\lambda_{f-1}(r_{f-1})\right)\in\mathbb{Z}^{f}.\] Recall the linear polynomial $e(\bm\lambda)\in\mathbb{Z}[x_{0},x_{1},\ldots,x_{f-1}]$ associated to $\bm\lambda\in\left(\mathbb{Z}\pm x\right)^{f}$ in \cite[\S 2]{BP}: 
\begin{linenomath*}
	\[
	e(\bm{\lambda})(x_{0},\ldots,x_{f-1}):=
	\begin{cases}
	\frac{1}{2}\left(\sum\limits_{j=0}^{f-1}p^{j}(x_{j}-\lambda_{j}(x_{j}))\right) & \text{if $\lambda_{f-1}(x_{f-1})\in\lbrace x_{f-1},x_{f-1}-1\rbrace$}, \\
	\frac{1}{2}\left(p^{f}-1+\sum\limits_{j=0}^{f-1}p^{j}(x_{j}-\lambda_{j}(x_{j}))\right) & \text{otherwise}.
	\end{cases}
	\]
\end{linenomath*}
Now let $\bm{r}=(r_{0},r_{1},\ldots,r_{f-1})\in\mathbb{Z}^{f}$ such that $1\leq r_{j}\leq p-3$ for all $j$, and consider the following generic weights of $\Gamma$ \begin{align*}
\sigma_{k}:=\bm\mu^{(k)}(\bm{r})\otimes\mathrm{det}^{e_{k}(\bm{r})}\hspace{2mm}\text{for all $0\leq k\leq l$,}
\end{align*} where \[
e_{0}(\bm{r}):=0\hspace{2mm}\text{and}\hspace{2mm}e_{k}(\bm{r}):=\sum_{j=0}^{k-1}e(g^{j}\bm\mu)(\bm\mu^{(j)}(\bm{r}))\hspace{2mm}\text{for all $1\leq k\leq l$}.\] 
It is shown in \cite[Lemma 1.4 and Theorem 1.6]{she22} that $\sigma_{l}=\sigma_{0}=\bm{r}$, $E(\sigma_{k},\sigma_{k-1}^{[s]})$ exists for all $1\leq k\leq l$, and $E(\sigma_{k},\sigma_{k-1}^{[s]})^{U}=\chi(\sigma_{k})\oplus\chi(\sigma_{k-1})^{s}$ for all $1\leq k\leq l$. In other words, $C:=\bigoplus_{k=1}^{l}E(\sigma_{k},\sigma_{k-1}^{[s]})$ is a \emph{cyclic module} of $\Gamma$ (see \cite[Definition 1.1]{she22}). Permuting the $f$-tuples of $\sigma_{k}$'s by the application of $g\in S_{f}$, we obtain another cyclic module of $\Gamma$. Indeed, let \begin{align*}
\sigma_{k}':=(g\bm\mu^{(k)})(\bm{r})\otimes\mathrm{det}^{e'_{k}(\bm{r})}
\hspace{2mm}\text{for all $0\leq k\leq l$},\hspace{2mm}
\end{align*} where \[e_{0}'(\bm{r}):=0\hspace{2mm}\text{and}\hspace{2mm}e_{k}'(\bm{r}):=\sum_{j=0}^{k-1}e(g^{j+1}\bm\mu)((g\bm\mu^{(j)})(\bm{r}))\hspace{2mm} \text{for all $1\leq k\leq l$}.\] 

\begin{lemma}\label{C'_is_cyclic_lemma} For all $1\leq k\leq l$, $E(\sigma_{k}',\sigma_{k-1}'^{[s]})$ exists, and $C':=\bigoplus_{k=1}^{l}E(\sigma'_{k},\sigma_{k-1}'^{[s]})$ is a multiplicity-free cyclic module of $\Gamma$. 
\end{lemma}
\begin{proof}
The arguments similar to those in the proof of \cite[Lemma 1.4 (3)]{she22} show that the integer $e'_{l}(\bm{r})$ is independent of $\bm{r}$ and is $0$ modulo $p^{f}-1$. Thus $\sigma'_{l}=\sigma'_{0}=\bm{r}$. Now the first graded piece  $\mathrm{gr^{1}_{cosoc}}(\mathrm{Ind}^{\Gamma}_{B}\chi(\sigma_{k-1}')^{s})$ of the cosocle filtration of $\mathrm{Ind}^{\Gamma}_{B}\chi(\sigma_{k-1}')^{s}$ is \[\bigoplus_{i=0}^{f-1}(g^{i}\bm\mu)(( g\bm\mu^{(k-1)})(\bm{r}))\otimes\mathrm{det}^{(g^{i}\bm\mu)(( g\bm\mu^{(k-1)})(\bm{r}))}\mathrm{det}^{e'_{k-1}(\bm{r})}.\] So, $g\bm\mu^{(k)}=g^{k}\bm\mu\circ g\bm\mu^{(k-1)}$ implies that $\sigma_{k}'\subseteq\mathrm{gr^{1}_{cosoc}}(\mathrm{Ind}^{\Gamma}_{B}\chi(\sigma_{k-1}')^{s})$ for all $1\leq k\leq l$. As a result, $E(\sigma_{k}',\sigma_{k-1}'^{[s]})$ exists for all $k$, and $E(\sigma_{k}',\sigma_{k-1}'^{[s]})^{U}=\chi(\sigma_{k}')\oplus\chi(\sigma_{k-1}')^{s}$. As the $f$-tuples $\{g\bm\mu^{(1)},g\bm\mu^{(2)},\ldots,g\bm\mu^{(l)}\}$ are all distinct, it follows that $C'$ is a cyclic module. The multiplicity-freeness  of $C$ implies that $C'$ is also multiplicity-free.
\end{proof}

Let $\sigma:=\sigma_{l}=\sigma_{l}'$ and $\sigma^{[s]}:=\sigma_{l}^{[s]}=\sigma_{l}'^{[s]}$. Note that $\sigma$ (resp. $\sigma^{[s]}$) occurs with multiplicity two in the socle (resp. cosocle) of $C\oplus C'$ while all the other socle (resp. cosocle) weights occur with multiplicity one by Lemma \ref{multiplicity_free_lemma}. We construct a certain subquotient of $C\oplus C'$ by splicing $C$ and $C'$ together along $\sigma$ and $\sigma^{[s]}$. The resulting spliced module will have multiplicity-free socle and cosocle. 

Let $\iota_{\sigma}$ and $\iota_{\sigma^{[s]}}$ be the compositions \[\sigma\xhookrightarrow{\Delta}\sigma\oplus\sigma\xhookrightarrow{}\mathrm{soc}_{\Gamma}\left(C\oplus C'\right)\hspace{3mm}\mathrm{and}\hspace{3mm}\sigma^{[s]}\xhookrightarrow{\Delta}\sigma^{[s]}\oplus\sigma^{[s]}\xhookrightarrow{}\mathrm{cosoc}_{\Gamma}\left(C\oplus C'\right)\hspace{3mm}\mathrm{respectively,}\] where the first map $\Delta$ in both is the diagonal embedding and the second map in both is the natural inclusion. As the cyclic modules $C$ and $C'$ are individually multiplicity-free (Lemma \ref{C'_is_cyclic_lemma}), $\sigma\not\in\mathrm{cosoc}_{\Gamma}\left(C\oplus C'\right)$ and $\sigma^{[s]}\not\in\mathrm{soc}_{\Gamma}\left(C\oplus C'\right)$. Thus, one has the following short exact sequence of $\Gamma$-modules \[0\longrightarrow\sigma\oplus\left(\bigoplus_{k=1}^{l-1}\sigma_{k}\oplus\sigma_{k}'\right)\longrightarrow\frac{C\oplus C'}{\iota_{\sigma}(\sigma)}\longrightarrow\mathrm{cosoc}_{\Gamma}\left(C\oplus C'\right)\longrightarrow 0.\] Define the spliced module $D_{0}$ to be the submodule of $\frac{C\oplus C'}{\iota_{\sigma}(\sigma)}$ that sits in the following short exact sequence \begin{equation}\label{D0_defn}
0\longrightarrow\sigma\oplus\left(\bigoplus_{k=1}^{l-1}\sigma_{k}\oplus\sigma_{k}'\right)\longrightarrow D_{0}\longrightarrow\iota_{\sigma^{[s]}}(\sigma^{[s]})\oplus\left(\bigoplus_{k=1}^{l-1}\sigma_{k}^{[s]}\oplus\sigma_{k}'^{[s]}\right)\longrightarrow 0.
\end{equation} 

The Hasse diagram of the cosocle filtration of $D_{0}$ looks as follows:
\begin{center}
\begin{tabular}{m{.6cm}m{.3cm}m{.3cm}m{.3cm}m{.5cm}m{.1cm}m{4cm}m{.1cm}m{.5cm}m{.3cm}m{.3cm}m{.1cm}m{.5cm}}
$\xymatrix{\sigma_{l-2}^{[s]}\ar@{-}[d] \\ \sigma_{l-1}}$ & \vspace{1.6cm}$\bigoplus$ & \vspace{1.6cm}$\ldots$ & \vspace{1.6cm}$\bigoplus$ & $\xymatrix{\sigma_{1}^{[s]}\ar@{-}[d] \\ \sigma_{2}}$ & \vspace{1.6cm}$\bigoplus$ & $\xymatrix{\sigma_{l-1}^{[s]}\ar@{-}[dr] & \sigma^{[s]}\ar@{-}[dl]\ar@{-}[dr] & \sigma_{l-1}'^{[s]}\ar@{-}[dl] \\ \sigma_{1} & \sigma & \sigma_{1}'}$ & \vspace{1.6cm}$\bigoplus$ & $\xymatrix{\sigma_{1}'^{[s]}\ar@{-}[d] \\ \sigma_{2}'}$ & \vspace{1.6cm}$\bigoplus$ & \vspace{1.6cm}$\ldots$ & \vspace{1.6cm}$\bigoplus$ & $\xymatrix{\sigma_{l-2}'^{[s]}\ar@{-}[d] \\ \sigma_{l-1}'}$
\end{tabular}
\end{center}
Notice that $D_{0}$ is a direct sum of $2(l-2)$ non-split extensions and two indecomposable modules of length 3 shown in the middle of the above diagram. Of these two indecomposable modules, let us denote the one with socle $\sigma$ by $M(\sigma)$ and the other one with cosocle $\sigma^{[s]}$ by $M(\sigma^{[s]})$. The module $M(\sigma)$ is a quotient of $E(\sigma,\sigma_{l-1}^{[s]})\oplus E(\sigma,\sigma_{l-1}'^{[s]})$ such that the natural surjection $E(\sigma,\sigma_{l-1}^{[s]})\oplus E(\sigma,\sigma_{l-1}'^{[s]})\twoheadrightarrow M(\sigma)$ restricted to individual extensions is an isomorphism. Similarly, the module $M(\sigma^{[s]})$ is a submodule of $E(\sigma_{1},\sigma^{[s]})\oplus E(\sigma_{1}',\sigma^{[s]})$ such that the natural maps $\frac{M(\sigma^{[s]})}{M(\sigma^{[s]})\cap E(\sigma_{1},\sigma^{[s]})}\rightarrow E(\sigma_{1}',\sigma^{[s]})$ and $\frac{M(\sigma^{[s]})}{M(\sigma^{[s]})\cap E(\sigma_{1}',\sigma^{[s]})}\rightarrow E(\sigma_{1},\sigma^{[s]})$ are isomorphisms. 
\begin{remark} Though the socle and the cosocle of $D_{0}$ are multiplicity-free by construction, $D_{0}$ need not be multiplicity-free. For example, when $f=2$, the weight $(p-2-r_{0},r_{1}+1)\otimes\mathrm{det}^{r_{0}+p(p-1)}$ occurs in the socle of $C$ as well as in the cosocle of $C'$.
\end{remark}

Let $D_{1}:=D_{0}^{U}$, $S_{1}:=\left(\mathrm{soc}_{\Gamma}D_{0}\right)^{U}$, and $Q_{1}:=\left(\mathrm{cosoc}_{\Gamma}D_{0}\right)^{U}$. The $B$-representations $S_{1}$ and $Q_{1}$ are multiplicity-free, i.e., for a $B$-character $\chi$, we have $\mathrm{dim}_{\overline{\mathbb{F}}_{p}}S_{1}^{\chi}\leq 1$ and $\mathrm{dim}_{\overline{\mathbb{F}}_{p}}Q_{1}^{\chi}\leq 1$.
 
\begin{lemma}\label{D1_characters_lemma} As $B$-representations, \[D_{1}=S_{1}\oplus Q_{1}=\chi(\sigma)\oplus\chi(\sigma)^{s}\oplus\left(\bigoplus_{k=1}^{l-1}\chi(\sigma_{k})\oplus\chi(\sigma_{k}')\oplus\chi(\sigma_{k})^{s}\oplus\chi(\sigma_{k}')^{s}\right).\] Thus, for a $B$-character $\chi$, $\mathrm{dim}_{\overline{\mathbb{F}}_{p}}S_{1}^{\chi}=1$ if and only if $\mathrm{dim}_{\overline{\mathbb{F}}_{p}}Q_{1}^{\chi^{s}}=1$.
\end{lemma}
\begin{proof}
The second part follows from the first part and the discussion before the lemma. The first part is equivalent to the claim that $\mathrm{dim}_{\overline{\mathbb{F}}_{p}}D_{1}=4l-2$ because $D_{0}$, by definition, has length $4l-2$ (\ref{D0_defn}). Note that $\mathrm{dim}_{\overline{\mathbb{F}}_{p}}(C\oplus C')^{U}=4l$ implies that $\mathrm{dim}_{\overline{\mathbb{F}}_{p}}\left(\frac{C\oplus C'}{\iota_{\sigma}(\sigma)}\right)^{U}\geq 4l-1$. However, the $\Gamma$-module $\frac{C\oplus C'}{\iota_{\sigma}(\sigma)}$ has length $4l-1$. Hence $\mathrm{dim}_{\overline{\mathbb{F}}_{p}}\left(\frac{C\oplus C'}{\iota_{\sigma}(\sigma)}\right)^{U}=4l-1$. Since $D_{0}$ sits in the short exact sequence \[0\longrightarrow D_{0}\longrightarrow\frac{C\oplus C'}{\iota_{\sigma}(\sigma)}\longrightarrow\sigma^{[s]}\longrightarrow 0\] and the functor of $U$-invariants is left exact, we have  \[\mathrm{dim}_{\overline{\mathbb{F}}_{p}}D_{1}+\mathrm{dim}_{\overline{\mathbb{F}}_{p}}\mathrm{Im}\left(\left(\frac{C\oplus C'}{\iota_{\sigma}(\sigma)}\right)^{U}\rightarrow\left(\sigma^{[s]}\right)^{U}\right)=\mathrm{dim}_{\overline{\mathbb{F}}_{p}}\left(\frac{C\oplus C'}{\iota_{\sigma}(\sigma)}\right)^{U}=4l-1.\] As $\mathrm{dim}_{\overline{\mathbb{F}}_{p}}\mathrm{Im}\left(\left(\frac{C\oplus C'}{\iota_{\sigma}(\sigma)}\right)^{U}\rightarrow\left(\sigma^{[s]}\right)^{U}\right)\leq 1$ and $\mathrm{dim}_{\overline{\mathbb{F}}_{p}}D_{1}\leq 4l-2$, the claim follows.
\end{proof}

\begin{remark}
	We remark that one can work with any two cyclic modules of $\Gamma$ arising from two different cyclic permutations of $\bm{\mu}$ to form a spliced module $D_{0}$ (see \cite[Remark 1.7]{she22}). 
\end{remark}

\begin{remark}
    Recently, M. Schein \cite{sch22} constructed interesting \emph{cyclic} diagrams built out of principal series to construct irreducible admissible supercuspidal representations of $G$ with $K$-socles compatible with Serre's weight conjecture in the ramified setting.
\end{remark}

\section{Infinite-dimensional irreducible diagram}

To construct diagrams in the sense of \cite[\S 9]{BP}, equip the spliced module $D_{0}$ with a smooth $KZ$-action via $KZ\twoheadrightarrow\Gamma$ such that $\varpi$ acts trivially. Equip $D_{1}$ with a smooth $N$-action by defining the action of $\Pi$ to be a linear automorphism of order 2 that maps $S_{1}^{\chi}$ to $Q_{1}^{\chi^{s}}$ for all $I$-characters $\chi$ such that $S_{1}^{\chi}\neq 0$ (see Lemma \ref{D1_characters_lemma}). This gives rise to a basic $0$-diagram $(D_{0},D_{1},\mathrm{can})$  where $\mathrm{can}:D_{1}\hookrightarrow D_{0}$ is the canonical inclusion (see \cite[Introduction, page 3]{BP} for the definition of a basic $0$-diagram). It is easy to see that the diagram $(D_{0},D_{1},\mathrm{can})$ is irreducible.

Let $D_{0}(\infty):=\bigoplus_{i\in\mathbb{Z}}D_{0}(i)$ be the smooth $KZ$-representation with component-wise $KZ$-action, where there is a fixed 
isomorphism $D_{0}(i)\cong D_{0}$ of $KZ$-representations for every $i\in\mathbb{Z}$. Following \cite{Le}, we denote the natural inclusion 
$D_{0}\iso D_{0}(i)\hookrightarrow D_{0}(\infty)$ by $\iota_{i}$, and write $v_{i}:=\iota_{i}(v)$ for $v\in D_{0}$ for every $i\in\mathbb{Z}$. 
Let $D_{1}(\infty):=D_{0}(\infty)^{I(1)}\cong\bigoplus_{i\in\mathbb{Z}}(S_{1}\oplus Q_{1})$. We define a $\Pi$-action on $D_{1}(\infty)$ as follows. 
Let $\lambda=(\lambda_{i})\in\prod_{i\in\mathbb{Z}}\overline{\mathbb{F}}_{p}^{\times}$. For all integers $i\in\mathbb{Z}$, define
\begin{equation*}
\Pi v_{i}:=
\begin{cases}
\lambda_{i}(\Pi v)_{i} &\text{if $v\in S_{1}^{\chi(\sigma)}$,}\\
(\Pi v)_{i-1} &\text{if $v\in S_{1}^{\chi(\sigma_{1})}$,}\\
(\Pi v)_{i+1} &\text{if $v\in S_{1}^{\chi(\sigma_{1}')}$,}\\
(\Pi v)_{i} &\text{if $v\in S_{1}^{\chi}$ for $\chi\in\{\chi(\sigma_{2}),\ldots ,\chi(\sigma_{l-1}),\chi(\sigma_{2}'),\ldots ,\chi(\sigma_{l-1}')\}$.}
\end{cases}
\end{equation*} 
This uniquely determines a smooth $N$-action on $D_{1}(\infty)$ such that $\varpi=\Pi^{2}$ acts trivially on it. Thus we get a basic $0$-diagram 
$D(\lambda):=(D_{0}(\infty),D_{1}(\infty),\mathrm{can})$ with the above actions where can is the canonical inclusion $D_{1}(\infty)\hookrightarrow D_{0}(\infty)$. 

\begin{proposition}\label{irreducible_diagram_prop} If $\lambda_{i-1}\lambda_{i}\neq\lambda_{-1}\lambda_{0}$ for all $i\neq 0$, then the basic $0$-diagram $D(\lambda)$ is irreducible. 
\end{proposition}
\begin{proof}
Let $W\subseteq D_{0}(\infty)$ be a non-zero $KZ$-subrepresentation such that $\Pi$ stabilizes $W^{I(1)}$. The claim is $W=D_{0}(\infty)$. We have $\mathrm{Hom}_{K}(\tau, W)\neq 0$ for some $\tau\in\mathrm{soc}_{K}D_{0}$. We first consider the case $\tau=\sigma$.  

There exists a non-zero $(c_{i})\in\bigoplus_{i\in\mathbb{Z}}\overline{\mathbb{F}}_{p}$ such that \[\left(\sum_{i}c_{i}\iota_{i}\right)(\sigma)\subseteq W.\] We pick $(c_{i})$ with $\#(c_{i}):=\#\{i\in\mathbb{Z}:c_{i}\neq 0\}$ minimal. We first show that $\#(c_{i})=1$. The $\Pi$-action on $\left(\sum_{i}c_{i}\iota_{i}\right)(S_{1}^{\chi(\sigma)})$ gives $\left(\sum_{i}\lambda_{i}c_{i}\iota_{i}\right)(Q_{1}^{\chi(\sigma)^{s}})\subseteq W^{I(1)}$ which implies that $\left(\sum_{i}\lambda_{i}c_{i}\iota_{i}\right)(M(\sigma^{[s]}))\subseteq W$ because $M(\sigma^{[s]})$ is indecomposable. Hence
\begin{equation}\label{eq1}
\left(\sum_{i}\lambda_{i}c_{i}\iota_{i}\right)(\sigma_{1}\oplus\sigma_{1}')\subseteq W.
\end{equation} Now the $\Pi$-action on $\left(\sum_{i}\lambda_{i}c_{i}\iota_{i}\right)(S_{1}^{\chi(\sigma_{1})})$ and $\left(\sum_{i}\lambda_{i}c_{i}\iota_{i}\right)(S_{1}^{\chi(\sigma_{1}')})$ gives respectively \[\left(\sum_{i}\lambda_{i}c_{i}\iota_{i-1}\right)(Q_{1}^{\chi(\sigma_{1})^{s}})\subseteq W^{I(1)}\hspace{2mm}\mathrm{and}\hspace{2mm}\left(\sum_{i}\lambda_{i}c_{i}\iota_{i+1}\right)(Q_{1}^{\chi(\sigma_{1}')^{s}})\subseteq W^{I(1)}.\] Hence \begin{equation*}
\left(\sum_{i}\lambda_{i}c_{i}\iota_{i-1}\right)(E(\sigma_{2},\sigma_{1}^{[s]}))\subseteq W\hspace{2mm}\mathrm{and}\hspace{2mm}\left(\sum_{i}\lambda_{i}c_{i}\iota_{i+1}\right)(E(\sigma_{2}',\sigma_{1}'^{[s]}))\subseteq W. 
\end{equation*} 
The cyclicity of the $\Pi$-action on $I$-characters of $C$ and $C'$ then gives respectively \begin{equation}\label{eq2}
\left(\sum_{i}\lambda_{i}c_{i}\iota_{i-1}\right)(E(\sigma_{k},\sigma_{k-1}^{[s]}))\subseteq W\hspace{2mm}\mathrm{and}\hspace{2mm}\left(\sum_{i}\lambda_{i}c_{i}\iota_{i+1}\right)(E(\sigma_{k}',\sigma_{k-1}'^{[s]}))\subseteq W 
\end{equation} for all $2\leq k\leq l$. Therefore
\begin{equation}\label{eq3}
\left(\sum_{i}\lambda_{i} c_{i}\iota_{i-1}\right)(\sigma)\subseteq W\hspace{2mm}\mathrm{and}\hspace{2mm}\left(\sum_{i}\lambda_{i} c_{i}\iota_{i+1}\right)(\sigma)\subseteq W.
\end{equation} Thus, by increasing or decreasing the index $i$ if needed, we may assume $c_{0}\neq 0$. Now, repeating the above argument for $\big(\sum_{i}\lambda_{i} c_{i}\iota_{i-1}\big)(\sigma)\subseteq W$, we obtain \[\left(\sum_{i}\lambda_{i-1}\lambda_{i} c_{i}\iota_{i-2}\right)(\sigma)\subseteq W\hspace{2mm}\mathrm{and}\hspace{2mm}\left(\sum_{i}\lambda_{i-1}\lambda_{i} c_{i}\iota_{i}\right)(\sigma)\subseteq W.\] Note that $\big(\sum_{i}\lambda_{-1}\lambda_{0} c_{i}\iota_{i}\big)(\sigma)\subseteq W$. So it follows that \[\left(\sum_{i}(\lambda_{i-1}\lambda_{i}-\lambda_{-1}\lambda_{0})c_{i}\iota_{i}\right)(\sigma)\subseteq W.\] Write $c_{i}':=(\lambda_{i-1}\lambda_{i}-\lambda_{-1}\lambda_{0})c_{i}$ so that $\big(\sum_{i}c_{i}'\iota_{i}\big)(\sigma)\subseteq W$. If $\#(c_{i})>1$, then the hypothesis on $(\lambda_{i})$ contradicts the minimality of $(c_{i})$ because $\#(c_{i}')=\#(c_{i})-1$. Therefore $c_0\iota_{0}(\sigma)\subseteq W$ and hence $\iota_{0}(\sigma)\subseteq W$.

Now we repeat the above argument for $\iota_{0}(\sigma)\subseteq W$ to show that $\iota_{0}(D_{0})\subseteq W$. Indeed, the $\Pi$-action on $\iota_{0}(S_{1}^{\chi(\sigma)})$ gives \[\iota_{0}(M(\sigma^{[s]}))\subseteq W.\] By (\ref{eq3}), we have \[\iota_{-1}(\sigma)\subseteq W\hspace{2mm}\mathrm{and}\hspace{2mm}\iota_{1}(\sigma)\subseteq W.\] Using (\ref{eq1}) for the above inclusions, we obtain \[\iota_{1}(\sigma_{1})\subseteq W\hspace{2mm}\mathrm{and}\hspace{2mm}\iota_{-1}(\sigma_{1}')\subseteq W,\] and then using (\ref{eq2}), we get \[\iota_{0}(E(\sigma_{k},\sigma_{k-1}^{[s]}))\subseteq W\hspace{2mm}\mathrm{and}\hspace{2mm}\iota_{0}(E(\sigma_{k}',\sigma_{k-1}'^{[s]}))\subseteq W\] for all $2\leq k\leq l$. Together with the inclusion $\iota_{0}(M(\sigma^{[s]}))\subseteq W,$ this gives \[\iota_{0}(D_{0})\subseteq W.\] 

Repeat the argument for $\iota_{-1}(\sigma)\subseteq W$ and $\iota_{1}(\sigma)\subseteq W$ to obtain $\bigoplus_{i=0, \pm 1}\iota_{i}(D_{0})\subseteq W$, and so on. This process eventually gives $\bigoplus_{i\in\mathbb{Z}}\iota_{i}(D_{0})=D_{0}(\infty)\subseteq W$. 

If $\mathrm{Hom}_{K}(\tau, W)\neq 0$ for $\tau\neq\sigma$, then using the cyclicity of the $\Pi$-action as above, we reduce to the case $\mathrm{Hom}_{K}(\sigma, W)\neq 0$. 
\end{proof}

\begin{remark}\label{main-idea-rmk} The main idea here to construct an infinite-dimensional irreducible diagram is to arrange the $\Pi$-action on the infinite sum of a spliced module so that the cycling on one loop increases the index and the cycling on the other decreases the index. This construction does not work for $\mathrm{GL}_{2}(F)$ when $F$ has residue degree $1$ because the cyclic modules of $\mathrm{GL}_{2}(\mathbb{F}_{p})$ are principal series representations, i.e., extensions of the form $E(\tau,\tau^{[s]})$, and principal series are too small to form spliced modules with two loops.  
\end{remark}

\section{Proofs of main theorems}

\begin{proof}[Proof of Theorem \ref{thm:intro} for $n=2$]
We first construct a desired representation $\pi$ of $G=\mathrm{GL}_{2}(F)$ over $\overline{\mathbb{F}}_{p}$. The construction is similar to that of \cite[Theorem 3.1]{Le} or \cite[Theorem 1]{GS}. Let $\Omega$ be the smooth injective $K$-envelope of $D_{0}$ equipped with the $KZ$-action such that $\varpi$ acts trivially. The smooth injective $I$-envelope $\mathrm{inj}_{I}D_{1}$ of $D_{1}$ is an $I$-direct summand of $\Omega$. Let $e$ denote the projection of $\Omega$ onto $\mathrm{inj}_{I}D_{1}$. 
There is a unique $N$-action on $\mathrm{inj}_{I}D_{1}$ compatible with that of $I$ and compatible with the action of $N$ on $D_{1}$. By \cite{BP}, Lemma 9.6, 
there is a non-canonical $N$-action on $(1-e)(\Omega)$ extending the given $I$-action. This gives an $N$-action on $\Omega$ whose restriction to $IZ$ is 
compatible with the action coming from $KZ$ on $\Omega$.

Let $D(\lambda)=(D_{0}(\infty),D_{1}(\infty),\mathrm{can})$ be an irreducible infinite-dimensional diagram from Proposition \ref{irreducible_diagram_prop}. Let $\Omega(\infty):=\bigoplus_{i\in\mathbb{Z}}\Omega(i)$ with component-wise $KZ$-action where there is a fixed isomorphism $\Omega(i)\cong\Omega$ of 
$KZ$-representations for every $i\in\mathbb{Z}$. As before, denote the natural inclusion $\Omega\iso\Omega(i)\hookrightarrow \Omega(\infty)$ by $\iota_{i}$, and write $v_{i}:=\iota_{i}(v)$ for $v\in\Omega$. Let $\Omega_{\chi}$ denote the smooth injective $I$-envelope of an $I$-character $\chi$. We have \[e(\Omega)=\mathrm{inj}_{I}D_{1}=\mathrm{inj}_{I}S_{1}\oplus\mathrm{inj}_{I}Q_{1}=\bigoplus\Omega_{S_{1}^{\chi}}\oplus\Omega_{Q_{1}^{\chi^{s}}}.\] 
If $v\in(1-e)(\Omega)$, we define $\Pi v_{i}:=(\Pi v)_{i}$ for all integers $i$. Otherwise, we define \begin{equation*}
\Pi v_{i}:=
\begin{cases}
\lambda_{i}(\Pi v)_{i} &\text{if $v\in \Omega_{S_{1}^{\chi(\sigma)}}$,}\\
(\Pi v)_{i-1} &\text{if $v\in \Omega_{S_{1}^{\chi(\sigma_{1})}}$,}\\
(\Pi v)_{i+1} &\text{if $v\in \Omega_{S_{1}^{\chi(\sigma_{1}')}}$,}\\
(\Pi v)_{i} &\text{if $v\in \Omega_{S_{1}^{\chi}}$ for $\chi\in\{\chi(\sigma_{2}),\ldots ,\chi(\sigma_{l-1}),\chi(\sigma_{2}'),\ldots ,\chi(\sigma_{l-1}')\}$.}
\end{cases}
\end{equation*} 
By demanding that $\Pi^{2}$ acts trivially, this defines a 
smooth $N$-action on $\Omega(\infty)$ which is compatible with the $N$-action on $D_{1}(\infty)$, and whose restriction to $IZ$ is compatible with the action 
coming from $KZ$ on $\Omega(\infty)$. By \cite[Corollary 5.5.5]{Paskunas}, there is a smooth $G$-action on $\Omega(\infty)$. Take $\pi$ to be the $G$-representation 
generated by $D_{0}(\infty)$ inside $\Omega(\infty)$. The smooth representation $\pi$ has a property that $\mathrm{soc}_{K}\pi=\mathrm{soc}_{K}D_{0}(\infty)$. Since $D(\lambda)$ is irreducible and $\mathrm{soc}_{K}D_{0}(\infty)$ is infinite-dimensional, it follows that $\pi$ is irreducible and non-admissible. 

Note that the spliced module $D_{0}$, the diagram $D(\lambda)$, and the module $\Omega(\infty)$ are all defined over the residue field $\mathbb{F}_{p^{f}}$ of $F$. Hence, if $(\lambda_{i})\in\prod_{i\in\mathbb{Z}}\mathbb{F}_{p^{f}}^{\times}$ and $\lambda_{i-1}\lambda_{i}\neq\lambda_{-1}\lambda_{0}$ for all $i\neq 0$, then the $G$-representation $\pi$ has a model $\pi_{0}$ over $\mathbb{F}_{p^{f}}$ that is absolutely irreducible and non-admissible. This gives Theorem 1.1 for $n=2$.
\end{proof}

\begin{proof}[Proof of Theorem \ref{thm:schur}]
Let $\pi$ be a non-admissible irreducible representation of \linebreak $\mathrm{GL}_{2}(F)$ over $\overline{\mathbb{F}}_{p}$ constructed in the proof of Theorem 1.1 for $n=2$. 
Let $(\mu_i) \in \prod_{i \geq 0}\overline{\mathbb{F}}_{p}^{\times}$ be an $\mathbb{F}_{p^f}$-basis for $\overline{\mathbb{F}}_{p}$ with $\mu_0 = 1$. 
Let $\lambda_{-1} = \lambda_0 = 1$. Define inductively $\lambda_i = \frac{\mu_i}{\mu_{i-1}\lambda_{i-1}}$ for $i>0$ and let $\lambda_{i} = \lambda_{-i-1}$ for $i<0$. 
Then the restriction of scalars of $\pi$ to $\mathbb{F}_{p^f}$ is irreducible and its endomorphism algebra contains $\overline{\mathbb{F}}_{p}$. 
Indeed, for the irreducibility, suppose that $W \subseteq \pi$ is a non-zero $\mathbb{F}_{p^f}[G]$-subrepresentation. We have $\mathrm{Hom}_{\mathbb{F}_{p^{f}}[K]}(\tau,W)\neq 0$ for some weight $\tau$ in the $\mathbb{F}_{p^{f}}[K]$-socle of $D_{0}$. Let $\sigma_{\mathbb{F}_{p^{f}}}$ denote an $\mathbb{F}_{p^{f}}$-model of $\sigma$. Using the $\Pi$-action, we may assume $\tau=\sigma_{\mathbb{F}_{p^{f}}}$. The proof of Proposition 3.1 shows that $c_{0}\iota_0(\sigma_{\mathbb{F}_{p^{f}}}) \subseteq W$ for some $c_{0}\in\overline{\mathbb{F}}_{p}^{\times}$. Moreover, repeated use of (3.3) gives $\lambda_{0}\lambda_{1}c_{0}\iota_0(\sigma_{\mathbb{F}_{p^{f}}})\subseteq W$ and $\lambda_0 \lambda_j \left(\prod_{i=1}^{j-1} \lambda_i^2\right) c_{0}\iota_0(\sigma_{\mathbb{F}_{p^{f}}}) \subseteq W$ for all $j >1$. 
Since $\lambda_{0}\lambda_{1}=\mu_{1}$ and since $\lambda_0 \lambda_j \prod_{i=1}^{j-1} \lambda_i^2  = \mu_j$ for $j>1$  by an easy induction on $j$, we have $c \iota_0(\sigma_{\mathbb{F}_{p^{f}}}) \subseteq W$ for all $c \in \overline{\mathbb{F}}_{p}$ and thus $\iota_{0}(\sigma)\subseteq W$. 
Then one proceeds as in the proof of Proposition 3.1 to show that $D_{0}(\infty)\subseteq W$ which implies that $W=\pi$. 
\end{proof}

\begin{proof}[Proof of Theorem \ref{thm:intro} for $n>2$]
Let $\mathrm{P}=\mathrm{M}\mathrm{N}$ be the standard parabolic subgroup of $\mathrm{GL}_{n}$ with Levi subgroup $\mathrm{M}=\mathrm{GL}_{2}\times(\mathrm{GL}_{1})^{n-2}$. Let $\overline{\mathrm{P}}=\mathrm{M}\overline{\mathrm{N}}$ be the opposite parabolic subgroup. Let $\rho$ be a non-admissible irreducible representation of $\mathrm{GL}_{2}(F)$ over $\overline{\mathbb{F}}_{p}$ constructed in the proof of Theorem \ref{thm:intro} for $n=2$, and let $\chi$ be a character of $(F^{\times})^{n-2}$. Consider the smooth irreducible non-admissible representation $\rho\otimes\chi$ of $\mathrm{M}(F)$, and let \[\pi=\mathrm{Ind}_{\overline{\mathrm{P}}(F)}^{\mathrm{GL}_{n}(F)}(\rho\otimes\chi)\] be the parabolically induced representation of $\mathrm{GL}_{n}(F)$. It is clear that $\pi$ is non-admissible because \[\pi^{K(1)}=\left(\mathrm{Ind}_{\overline{\mathrm{P}}(\mathcal{O}_{F})}^{\mathrm{GL}_{n}(\mathcal{O}_{F})}(\rho\otimes\chi)\right)^{K(1)}=\mathrm{Ind}_{\overline{\mathrm{P}}(\mathbb{F}_{p^{f}})}^{\mathrm{GL}_{n}(\mathbb{F}_{p^{f}})}\left((\rho\otimes\chi)^{M(1)}\right)\] and the latter is not finite-dimensional. Here, $K(1)=\text{Ker}(\mathrm{GL}_{n}(\mathcal{O}_{F})\twoheadrightarrow\mathrm{GL}_{n}(\mathbb{F}_{p^{f}}))$ and $M(1)=\text{Ker}(\mathrm{M}(\mathcal{O}_{F})\twoheadrightarrow\mathrm{M}(\mathbb{F}_{p^{f}}))$.

Recall that for a Levi subgroup $\mathrm{L}\subseteq\mathrm{GL}_{n}$, an $\mathrm{L}(\mathcal{O}_{F})$-weight is, by definition, a smooth irreducible representation of $\mathrm{L}(\mathcal{O}_{F})$. The endomorphism algebra $\mathrm{End}_{\mathrm{L}(F)}(\text{c-Ind}_{\mathrm{L}(\mathcal{O}_{F})}^{\mathrm{L}(F)}\tau)$ of the compactly induced representation $\text{c-Ind}_{\mathrm{L}(\mathcal{O}_{F})}^{\mathrm{L}(F)}\tau$ of an $\mathrm{L}(\mathcal{O}_{F})$-weight $\tau$ is called the spherical algebra of $\mathrm{L}(F)$ and is denoted by $\mathcal{H}_{\mathrm{L}(F)}(\tau)$. For a smooth representation $V$ of $\mathrm{L}(F)$, an $\mathrm{L}(\mathcal{O}_{F})$-weight of $V$ simply means a smooth irreducible $\mathrm{L}(\mathcal{O}_{F})$-subrepresentation of $V.$ 

\begin{lemma}\label{lemma1-irreducibility}
If every $\mathrm{GL}_{n}(\mathcal{O}_{F})$-weight of $\pi$ is $\mathrm{M}$-regular (in the sense of \cite[Definition 2.4]{Her}), then $\pi$ is irreducible.
\end{lemma}
\begin{proof}
Let $\tau$ be a (non-zero) $\mathrm{GL}_{n}(\mathcal{O}_{F})$-weight of $\pi$. We will show that $\tau$ generates $\pi$ as a $\mathrm{GL}_{n}(F)$-representation. By Frobenius reciprocity, the canonical inclusion $\tau\hookrightarrow\pi\big|_{\mathrm{GL}_{n}(\mathcal{O}_{F})}$ corresponds to an injection $\tau^{\mathrm{N}(\mathbb{F}_{p^{f}})}\hookrightarrow(\rho\otimes\chi)\big|_{M(\mathcal{O}_{F})}$ which makes   $\tau^{\mathrm{N}(\mathbb{F}_{p^{f}})}$ into an $\mathrm{M}(\mathcal{O}_{F})$-weight of $\rho\otimes\chi$, cf. \cite[Lemma 2.3 and (2.13)]{Her}. Let $\tau_{\rho}:=\tau^{\mathrm{N}(\mathbb{F}_{p^{f}})}\big|_{\mathrm{GL}_{2}(\mathcal{O}_{F})}$ and $\chi_{0}:=\chi\big|_{(\mathcal{O}_{F}^{\times})^{n-2}}$ so that $\tau\cong\tau_{\rho}\otimes\chi_{0}$. The spherical Hecke algebra $\mathcal{H}_{\mathrm{M}(F)}(\tau^{\mathrm{N}(\mathbb{F}_{p^{f}})})$ of $\mathrm{M}(F)$ is isomorphic to the tensor product $\mathcal{H}_{\mathrm{GL}_{2}(F)}(\tau_{\rho})\otimes\mathcal{H}_{(F^{\times})^{n-2}}(\chi_{0})$ of the spherical Hecke algebras of $\mathrm{GL}_{2}(F)$ and $(F^{\times})^{n-2}$. The algebra $\mathcal{H}_{\mathrm{GL}_{2}(F)}(\tau_{\rho})$ is commutative by \cite[Proposition 8 (1)]{BL} and the algebra $\mathcal{H}_{(F^{\times})^{n-2}}(\chi_{0})$ is commutative by \cite[\S 2.10]{HV2}. Hence, the algebra $\mathcal{H}_{\mathrm{M}(F)}(\tau^{\mathrm{N}(\mathbb{F}_{p^{f}})})$ is commutative. Under Frobenius reciprocity, the injection $\tau^{\mathrm{N}(\mathbb{F}_{p^{f}})}\hookrightarrow(\rho\otimes\chi)\big|_{M(\mathcal{O}_{F})}$ corresponds to a map $f\in\mathrm{Hom}_{\mathrm{M}(F)}\left(\text{c-Ind}_{\mathrm{M}(\mathcal{O}_{F})}^{\mathrm{M}(F)}\tau^{\mathrm{N}(\mathbb{F}_{p^{f}})}, \rho\otimes\chi\right)$. We claim that $f$ is an eigenvector for the action of $\mathcal{H}_{\mathrm{M}(F)}(\tau^{\mathrm{N}(\mathbb{F}_{p^{f}})})$ on $\mathrm{Hom}_{\mathrm{M}(F)}\left(\text{c-Ind}_{\mathrm{M}(\mathcal{O}_{F})}^{\mathrm{M}(F)}\tau^{\mathrm{N}(\mathbb{F}_{p^{f}})}, \rho\otimes\chi\right)$. Indeed, the restriction of the injection $\tau^{\mathrm{N}(\mathbb{F}_{p^{f}})}\hookrightarrow(\rho\otimes\chi)\big|_{M(\mathcal{O}_{F})}$ to $\mathrm{GL}_{2}(\mathcal{O}_{F})$ gives a map $f_{\rho}\in\mathrm{Hom}_{\mathrm{GL}_{2}(F)}\left(\text{c-Ind}_{\mathrm{GL}_{2}(\mathcal{O}_{F})}^{\mathrm{GL}_{2}(F)}\tau_{\rho}, \rho\right)$. 
It is enough to show that $f_{\rho}$ is an eigenvector for the action of $\mathcal{H}_{\mathrm{GL}_{2}(F)}(\tau_{\rho})$ on $\mathrm{Hom}_{\mathrm{GL}_{2}(F)}\left(\text{c-Ind}_{\mathrm{GL}_{2}(\mathcal{O}_{F})}^{\mathrm{GL}_{2}(F)}\tau_{\rho}, \rho\right)$. The Hecke algebra $\mathcal{H}_{\mathrm{GL}_{2}(F)}(\tau_{\rho})$ is isomorphic to the polynomial algebra $\overline{\mathbb{F}}_{p}[S^{\pm 1},T]$ where the Hecke operators $S$ and $T$ correspond to the characteristic functions supported on $\mathrm{GL}_{2}(\mathcal{O}_{F})\left(\begin{smallmatrix}\varpi & 0 \\0 & \varpi\end{smallmatrix}\right)\mathrm{GL}_{2}(\mathcal{O}_{F})$ and $\mathrm{GL}_{2}(\mathcal{O}_{F})\left(\begin{smallmatrix}1 & 0 \\0 & \varpi\end{smallmatrix}\right)\mathrm{GL}_{2}(\mathcal{O}_{F})$ respectively. Since $\rho$ has central character, $f_{\rho}$ is an eigenvector for the operator $S$. We now show that $T\cdot f_{\rho}=f_{\rho}\circ T=0$. By \cite[Lemma 2.1]{sch22}, $f_{\rho}(T(\tau_{\rho}))$ is contained in a $K$-subrepresentation $W$ of $\rho$ generated by $\Pi v$ for a non-zero $v\in\tau_{\rho}^{I(1)}$. As $\rho$ is constructed from a spliced module, $W$ has length at most $3$ (see the Hasse diagram). On the other hand, $W$ naturally receives a surjection from $\mathrm{Ind}_{I}^{K}\chi(\tau_{\rho})^{s}$ which is multiplicity-free of length at least $4$ (as $f>1$) and has socle isomorphic to $\tau_{\rho}$, cf. \cite[Theorem 2.4]{BP}. Therefore $\tau_{\rho}$ is not a Jordan-H\"{o}lder factor of $W$. Hence $f_{\rho}(T(\tau_{\rho}))=0$. As $f_{\rho}$ and $T$ are $G$-equivariant, $T\cdot f=0$ on $\text{c-Ind}_{\mathrm{GL}_{2}(\mathcal{O}_{F})}^{\mathrm{GL}_{2}(F)}\tau_{\rho}$. This finishes the proof of the claim. 

The set of eigenvalues of $f$ gives a character $\psi:\mathcal{H}_{\mathrm{M}(F)}(\tau^{\mathrm{N}(\mathbb{F}_{p^{f}})})\rightarrow\overline{\mathbb{F}}_{p}$
and a surjective map
\begin{equation}\label{map1}
\text{c-Ind}_{\mathrm{M}(\mathcal{O}_{F})}^{\mathrm{M}(F)}\tau^{\mathrm{N}(\mathbb{F}_{p^{f}})}\otimes_{\mathcal{H}_{\mathrm{M}(F)}(\tau^{\mathrm{N}(\mathbb{F}_{p^{f}})}),\psi}\overline{\mathbb{F}}_{p}\twoheadrightarrow\rho\otimes\chi
\end{equation} of $\mathrm{M}(F)$-representations. Further, as $\tau$ is $\mathrm{M}$-regular, there is a natural isomorphism 
\begin{equation}\label{map2}
\text{c-Ind}_{\mathrm{GL}_{n}(\mathcal{O}_{F})}^{\mathrm{GL}_{n}(F)}\tau\otimes_{\mathcal{H}_{\mathrm{GL}_{n}(F)}(\tau),\psi}\overline{\mathbb{F}}_{p}\iso\text{Ind}_{\overline{\mathrm{P}}(F)}^{\mathrm{GL}_{n}(F)}\left(\text{c-Ind}_{\mathrm{M}(\mathcal{O}_{F})}^{\mathrm{M}(F)}\tau^{\mathrm{N}(\mathbb{F}_{p^{f}})}\otimes_{\mathcal{H}_{\mathrm{M}(F)}(\tau^{\mathrm{N}(\mathbb{F}_{p^{f}})}),\psi}\overline{\mathbb{F}}_{p}\right)
\end{equation} of $\mathrm{GL}_{n}(F)$-representations by  \cite[Theorem 3.1]{Her}. Therefore, (\ref{map1}) and (\ref{map2}) together give a surjective map \begin{equation}\label{map3}
\text{c-Ind}_{\mathrm{GL}_{n}(\mathcal{O}_{F})}^{\mathrm{GL}_{n}(F)}\tau\otimes_{\mathcal{H}_{\mathrm{GL}_{n}(F)}(\tau),\psi}\overline{\mathbb{F}}_{p}\twoheadrightarrow\pi
\end{equation} of $\mathrm{GL}_{n}(F)$-representations because $\text{Ind}_{\overline{\mathrm{P}}(F)}^{\mathrm{GL}_{n}(F)}$ is exact. Since $\tau$ generates the left-hand side of (\ref{map3}) as a $\mathrm{GL}_{n}(F)$-representation, it also generates $\pi$ as a $\mathrm{GL}_{n}(F)$-representation. 

Now, if $\pi'\subseteq\pi$ is a non-zero subrepresentation, then $\pi'$ contains a (non-zero) $\mathrm{GL}_{n}(\mathcal{O}_{F})$-weight. By the previous paragraph, this weight generates $\pi$ as a $\mathrm{GL}_{n}(F)$-representation. Hence $\pi'=\pi$. 
\end{proof}

\begin{lemma}\label{lemma2-irreducibility}
	There exists a smooth character $\chi$ of $(F^{\times})^{n-2}$ such that $\pi=\mathrm{Ind}_{\overline{\mathrm{P}}(F)}^{\mathrm{GL}_{n}(F)}(\rho\otimes\chi)$ is irreducible. 
\end{lemma}
\begin{proof}
We use the notation $F(a_{1},a_{2},\ldots,a_{n})$ in \cite[\S 3.3]{Her-thesis} to denote weights. By Lemma \ref{lemma1-irreducibility}, it suffices to show that there exists a smooth character $\chi$ of $(F^{\times})^{n-2}$ such that every $\mathrm{GL}_{n}(\mathcal{O}_{F})$-weight of $\pi$ is $\mathrm{M}$-regular. We pick $0\leq a,b<p^{f}-1$ such that $a\neq b$ and $a$ is different from all the determinant powers of weights in $\mathrm{soc}_{\mathrm{GL}_{2}(\mathcal{O}_{F})}\rho$. Such an $a$ exists because there are at most $4f-1$ distinct weights in $\mathrm{soc}_{\mathrm{GL}_{2}(\mathcal{O}_{F})}\rho=\mathrm{soc}_{\mathrm{GL}_{2}(\mathcal{O}_{F})}D_{0}(\infty)$, and $p^{f}-1>4f-1$ for $p>3$ and $f>1$. Consider the alternating tensor product $\chi_{0}=F(a)\otimes F(b)\otimes F(a)\ldots$ of $F(a)$ and $F(b)$ as a character of $(\mathcal{O}_{F}^{\times})^{n-2}$, and let $\chi$ be a character of $(F^{\times})^{n-2}$ such that $\chi|_{(\mathcal{O}_{F}^{\times})^{n-2}}=\chi_{0}$. We claim that $\chi$ works. Indeed, let $\tau=F(a_{1},\ldots,a_{n})$ be a $\mathrm{GL}_{n}(\mathcal{O}_{F})$-weight of $\pi$ with $p^{f}-1\geq a_{i}-a_{i+1}\geq 0$ for all $i$. Note that $\tau$ is $\mathrm{M}$-regular if and only if $a_{2}, a_{3}, \ldots, a_{n}$ are distinct, cf. the paragraph after \cite[Definition 2.4]{Her}. Since $\tau^{\mathrm{N}(\mathbb{F}_{p^{f}})}=F(a_{1},a_{2})\otimes F(a_{3})\otimes\ldots\otimes F(a_{n})$ is an $\mathrm{M}(\mathcal{O}_{F})$-weight of $\rho\otimes\chi$, we find that $a_{2}$ modulo $p^{f}-1$ is the determinant power of a weight in $\mathrm{soc}_{\mathrm{GL}_{2}(\mathcal{O}_{F})}\rho$, and for $i\geq 3$, $a_{i}\equiv a\mod{p^{f}-1}$ (resp. $b\mod{p^{f}-1}$) if $i$ is odd (resp. even). By the construction of $\chi$, we have $a_{i}\not\equiv a_{i+1}\mod{p^{f}-1}$ for all $2\leq i\leq n-1$. As the sequence $a_{2}, a_{3},\ldots ,a_{n}$ is decreasing, this implies that $a_{i}\neq a_{j}$ for all $2\leq i,j\leq n$ and $i\neq j$. 
\end{proof} 

We now take $\chi$ as in the proof of Lemma \ref{lemma2-irreducibility}. Then it follows from Lemma \ref{lemma2-irreducibility} that $\mathrm{GL}_{n}(F)$ admits a smooth irreducible non-admissible representation $\pi=\mathrm{Ind}_{\overline{\mathrm{P}}(F)}^{\mathrm{GL}_{n}(F)}(\rho\otimes\chi)$ over $\overline{\mathbb{F}}_{p}$. As explained in the proof of Theorem \ref{thm:intro} for $n=2$, the $\mathrm{GL}_{2}(F)$-representation $\rho$ can be chosen to have a model $\rho_{0}$ over $\mathbb{F}_{p^{f}}$. Then \[\pi_{0}=\mathrm{Ind}_{\overline{\mathrm{P}}(F)}^{\mathrm{GL}_{n}(F)}(\rho_{0}\otimes\chi)\] is a model of $\pi$ over $\mathbb{F}_{p^{f}}$ because $\text{Ind}_{\overline{\mathrm{P}}(F)}^{\mathrm{GL}_{n}(F)}$ commutes with scalar extension \cite[Proposition III.12 (i)]{HV}. It is clear that $\pi_{0}$ is absolutely irreducible and non-admissible. 
\end{proof}

\begin{remark}
	We remark that the methods of this note to construct non-admissible irreducible representations also apply to other connected split reductive groups $\mathbb{G}$ whenever $\mathbb{G}$ contains $\mathrm{GL}_{2}$ as a Levi subgroup, e.g., $\mathrm{GSp}_{4}$ or $\mathrm{G}_{2}$.
\end{remark}  
\bibliography{nonadm}

\providecommand{\bysame}{\leavevmode\hbox to3em{\hrulefill}\thinspace}
\providecommand{\MR}{\relax\ifhmode\unskip\space\fi MR }
\providecommand{\MRhref}[2]{%
  \href{http://www.ams.org/mathscinet-getitem?mr=#1}{#2}
}
\providecommand{\href}[2]{#2}
\begin{thebibliography}{LLHLM20}

\bibitem[AHHV17]{ahhv}
Noriyuki Abe, Guy Henniart, Florian Herzig, and Marie-France Vign\'{e}ras,
  \emph{Questions on mod p representations of reductive p-adic groups}, 2017,
  available at \url{https://arxiv.org/abs/1703.02063}.

\bibitem[Ber74]{Bernstein}
Joseph Bernstein, \emph{All reductive $p$-adic groups are of type {I}},
  Funkcional. Anal. i Prilo\v{z}en. \textbf{8} (1974), no.~2, 3--6.
  \MR{0348045}

\bibitem[Ber12]{Berger}
Laurent Berger, \emph{Central characters for smooth irreducible modular
  representations of {${\rm GL}_2({\bf Q}_p)$}}, Rend. Semin. Mat. Univ. Padova
  \textbf{128} (2012), 1--6 (2013). \MR{3076828}

\bibitem[BL94]{BL}
Laure Barthel and Ron Livn\'{e}, \emph{Irreducible modular representations of
  {${\rm GL}_2$} of a local field}, Duke Math. J. \textbf{75} (1994), no.~2,
  261--292. \MR{1290194}

\bibitem[BP12]{BP}
Christophe Breuil and Vytautas Pa\v{s}k\={u}nas, \emph{Towards a modulo {$p$}
  {L}anglands correspondence for {${\rm GL}_2$}}, Mem. Amer. Math. Soc.
  \textbf{216} (2012), no.~1016, vi+114. \MR{2931521}

\bibitem[Bre03]{Breuil}
Christophe Breuil, \emph{Sur quelques repr\'{e}sentations modulaires et
  {$p$}-adiques de {${\rm GL}_2(\bold Q_p)$}. {I}}, Compositio Math.
  \textbf{138} (2003), no.~2, 165--188. \MR{2018825}

\bibitem[EGH]{EGH}
Matthew Emerton, Toby Gee, and Eugen Hellmann, \emph{An introduction to the
  categorical p-adic {L}anglands program}, available at
  \url{https://arxiv.org/abs/2210.01404}.

\bibitem[GS20]{GS}
Eknath Ghate and Mihir Sheth, \emph{On non-admissible irreducible modulo {$p$}
  representations of {${\rm GL}_2(\Bbb {Q}_{p^2})$}}, C. R. Math. Acad. Sci.
  Paris \textbf{358} (2020), no.~5, 627--632. \MR{4149863}

\bibitem[GS22]{Ghate2022}
\bysame, \emph{Diagrams and mod $p$ representations of $p$-adic groups},
  Perfectoid Spaces, Infosys Sci. Found. Ser., Springer Nature Singapore, 2022,
  pp.~37--50. \MR{4433441}

\bibitem[HC70]{HC}
Harish-Chandra, \emph{Harmonic analysis on reductive {$p$}-adic groups},
  Lecture Notes in Mathematics, Vol. 162, Springer-Verlag, Berlin-New York,
  1970, Notes by G. van Dijk. \MR{0414797}

\bibitem[Her09]{Her-thesis}
Florian Herzig, \emph{The weight in a {S}erre-type conjecture for tame
  {$n$}-dimensional {G}alois representations}, Duke Math. J. \textbf{149}
  (2009), no.~1, 37--116. \MR{2541127}

\bibitem[Her11]{Her}
\bysame, \emph{The classification of irreducible admissible mod $p$
  representations of a $p$-adic $\mathrm{GL}_n$}, Invent. Math. \textbf{186}
  (2011), no.~2, 373--434. \MR{2845621}

\bibitem[HV15]{HV2}
Guy Henniart and Marie-France Vign\'{e}ras, \emph{A {S}atake isomorphism for
  representations modulo {$p$} of reductive groups over local fields}, J. Reine
  Angew. Math. \textbf{701} (2015), 33--75. \MR{3331726}

\bibitem[HV19]{HV}
\bysame, \emph{Representations of a {$p$}-adic group in characteristic {$p$}},
  Representations of reductive groups, Proc. Sympos. Pure Math., vol. 101,
  Amer. Math. Soc., Providence, RI, 2019, pp.~171--210. \MR{3930018}

\bibitem[Jac75]{Jacquet}
Herv\'{e} Jacquet, \emph{Sur les repr\'{e}sentations des groupes r\'{e}ductifs
  {$p$}-adiques}, C. R. Acad. Sci. Paris S\'{e}r. A-B \textbf{280} (1975), Aii,
  A1271--A1272. \MR{369624}

\bibitem[Le19]{Le}
Daniel Le, \emph{On some nonadmissible smooth irreducible representations for
  {$\rm GL_2$}}, Math. Res. Lett. \textbf{26} (2019), no.~6, 1747--1758.
  \MR{4078693}

\bibitem[LLHLM20]{LLLM}
Daniel Le, Bao~V. Le~Hung, Brandon Levin, and Stefano Morra, \emph{Serre
  weights and {B}reuil's lattice conjecture in dimension three}, Forum Math. Pi
  \textbf{8} (2020), e5, 135. \MR{4079756}

\bibitem[Pas04]{Paskunas}
Vytautas Pa\v{s}k\={u}nas, \emph{Coefficient systems and supersingular
  representations of {${\rm GL}_2(F)$}}, M\'{e}m. Soc. Math. Fr. (N.S.) (2004),
  no.~99, vi+84. \MR{2128381}

\bibitem[Sch23]{sch22}
Michael~M. Schein, \emph{A family of irreducible supersingular representations
  of {{\(\mathrm{GL}_2(F)\)}} for some ramified {{\(p\)}}-adic fields}, Isr. J.
  Math. \textbf{255} (2023), no.~2, 911--930 (English).

\bibitem[She22]{she22}
Mihir Sheth, \emph{On irreducible supersingular representations of {${\rm
  GL}_2(F)$}}, Pacific J. Math. \textbf{321} (2022), no.~2, 431--442.
  \MR{4562577}

\bibitem[Vig96]{Vigneras}
Marie-France Vign\'{e}ras, \emph{Repr\'{e}sentations {$l$}-modulaires d'un
  groupe r\'{e}ductif {$p$}-adique avec {$l\ne p$}}, Progress in Mathematics,
  vol. 137, Birkh\"{a}user Boston, Inc., Boston, MA, 1996. \MR{1395151}

\end{thebibliography}
\bibliographystyle{amsalpha}
\end{document}